\documentclass[12pt]{amsart}

\usepackage[letterpaper, margin=1in]{geometry}

\usepackage{amsmath,amssymb,amsthm,enumerate,amsfonts,nicefrac,fancyhdr,hyperref,graphicx,adjustbox,setspace,float}
\hypersetup{colorlinks=true, allcolors=blue} 

\newcommand{\inv}[1]{#1^{-1}}
\newcommand{\Z}{\mathbb{Z}}
\newcommand{\rarr}{\rightarrow}
\newcommand{\sube}{\subseteq}
\newcommand{\sub}{\subset}
\newcommand{\Mod}{\operatorname{Mod}}
\newcommand{\Aut}{\operatorname{Aut}}
\newcommand{\gpoid}{\pi_1(\Sigma, X)}
\newcommand{\gp}{\pi_1(\Sigma, x_0)}
\newcommand{\BCAut}{\operatorname{BCAut}}
\newcommand{\Out}{\operatorname{Out}}
\newcommand{\Homeo}{\operatorname{Homeo}}
\newcommand{\Inn}{\operatorname{Inn}}
\newcommand{\PAut}{\operatorname{PAut}}
\newtheorem{lemma}{Lemma}[subsection]
\newtheorem{theorem}{Theorem}[section]

\begin{document}
\nocite{*}

\title{The Dehn-Nielsen-Baer Theorem for bounded surfaces}
\begin{abstract}
    Let $\Sigma$ be a bounded surface. We prove the Dehn-Nielsen-Baer theorem for bounded surfaces to show that the mapping class group of $\Sigma$ is isomorphic to the automorphisms of the fundamental groupoid of $\Sigma$ that fix loops around the boundary.
\end{abstract}
\author{Elysia Wang}
\maketitle

\section{Introduction}
The \textit{mapping class group}, $\Mod(\Sigma)$, is the group of isotopy classes of orientation-preserving homeomorphisms of $\Sigma$, a surface without boundary. 
The \textit{extended mapping class group}, $\Mod^\pm(\Sigma)$, includes orientation-reversing homeomorphisms. 
The \textit{outer automorphism group} of the \textit{fundamental group} of $\Sigma$ is denoted as $\Out(\pi_1(\Sigma))$.
The Dehn-Nielsen-Baer Theorem tells us that $\Mod^\pm(\Sigma) \cong \Out(\pi_1(\Sigma))$, building a bridge between a purely topological object and a purely algebraic object. \\

An example of a useful application is in the study of $\Out(F_n)$, the outer automorphism group of the free group on $n$ generators. Let $\Gamma_n$ denote the bouquet of $n$ circles. The fundamental group of $\Gamma_n$ is isomorphic to $F_n$. In analogy to the Dehn-Nielsen-Baer theorem, we can study $\Out(F_n)$ 
by studying the mapping class group of $\Gamma_n$, since the group of homotopy equivalences of $\Gamma_n$ up to homotopy is isomorphic to $\Out(F_n)$. \\

But, what changes when we consider surfaces with boundary? If we allow boundary components on our surface, the fundamental group is no longer able to capture all of the topological information; $\pi_1(\Sigma, x_0)$ is not able to detect a Dehn twist around a boundary component, for $x_0$ not on the boundary.
We can resolve this by adding multiple basepoints to our surface and considering the \textit{fundamental groupoid}. By placing a basepoint on each boundary component, the fundamental groupoid is able to detect the information that the fundamental group would miss with just one basepoint. 
We will call automorphisms which stabilise any element represented by an arc entirely contained in a boundary component the \textit{boundary-constant automorphisms} of the fundamental groupoid,
$\BCAut(\gpoid)$, (see Section \ref{infbasepoints} for the formal definition). 
When we consider $\BCAut(\gpoid)$, we are able to identify these automorphisms of the fundamental groupoid with mapping class group elements. 
The requirement for automorphisms to be boundary-constant follows naturally from the fact that mapping class group elements fix the boundary pointwise. 
These are the only restrictions we need, which leads us to the main result:\\

\noindent
\textbf{Theorem 1.}
\textit{Let} $\Sigma$ \textit{be a bounded surface and let} $X\sub \partial\Sigma$ \textit{be a subset of the boundary with at least one point on each boundary component. Then, the map}
$$\Phi: \Mod(\Sigma) \rarr \BCAut(\gpoid),$$
\textit{given by the action of the mapping class group on the fundamental groupoid, is an isomorphism}

\subsection{Outline of the paper}
In Section 2, we discuss the Dehn-Nielsen-Baer theorem for surfaces without boundary and for bounded surfaces where homeomorphisms are not required to fix the boundary pointwise.
Section 3 contains facts about groupoids, the fundamental groupoid, and groupoid automorphisms, as well as proofs for some relevant lemmas.
The main theorem is proved in Section 4.

\subsection{Acknowledgements}
I was supported by Ross Willard's NSERC Discovery Grant no. RGPIN-2019-03931 during the undergraduate research term in which this research took place. I would like to thank my supervisor, Tyrone Ghaswala, for his guidance and support; this paper would not exist without him.

\section{The mapping class group}
Let $\Sigma$ be a genus $g\ge 0$ surface with $b\ge0$ boundary components.
Let $\Homeo^+(\Sigma, \partial\Sigma)$ denote the group of orientation-preserving
homeomorphisms of $\Sigma$ that restrict to the identity on $\partial\Sigma$.\\

The \textit{mapping class group} of $\Sigma$, denoted $\Mod(\Sigma)$, 
is the group of isotopy classes of elements of $\Homeo^+(\Sigma, \partial\Sigma)$,
where isotopies are required to fix the boundary pointwise.
Let $\Homeo_0(\Sigma, \partial\Sigma)$ denote the group of homeomorphisms isotopic to the identity.
Then, we can write
$$\Mod(\Sigma)
= \Homeo^+(\Sigma, \partial\Sigma)/\Homeo_0(\Sigma,\partial\Sigma).$$
Elements of Mod$(\Sigma)$ are called \textit{mapping classes} and we apply 
mapping classes right to left.

\subsection{The Dehn-Nielsen-Baer Theorem}
The Dehn-Nielsen-Baer Theorem is a powerful result for 
surfaces without boundary. 
It ties together the mapping class group—a purely topological object—with
the outer automorphism group of the fundamental group—a purely algebraic object.
It is this theorem that we will extend upon. \\

For a surface without boundary, $S$, let $\Mod^\pm(S)$ denote the 
\textit{extended mapping class group} of $S$. $\Mod^\pm(S)$ is the group of isotopy classes of
all homeomorphisms of $S$, including orientation-reversing ones. \\

The \textit{outer automorphism group} of a group $G$, denoted $\Out(G)$, is the group of
automorphisms of $G$ up to conjugation. Formally, we have
$\Out(G) = \Aut(G)/\Inn(G)$, where $\Aut(G)$ is the group of all automorphisms of $G$
and $\Inn(G)$ is the \textit{inner automorphism group} of $G$—the group of automorphisms
which map $g\mapsto hg\inv{h}$ for all $g\in G$, for some $h\in G$. \\

Let $x\in S$, let $\phi: S\rarr S$ be a continuous map, and take a path $\gamma$ from $x$ to $\phi(x)$. We get a homomorphism $\phi_*:\pi_1(S, x) \rarr \pi_1(S, x)$, given as follows. Let $\alpha$ be a loop based at $x$ and set $\phi_*([\alpha]) = [\gamma \cdot \phi(\alpha) \cdot \gamma^{-1}]$. Then, by fixing $\phi$, different choices of $\gamma$ give maps $\phi_*$ that differ by conjugation. 
Similarly, different choices of basepoint $x$ give maps $\phi_*$ that only differ by conjugation, since we can let $\gamma'$ be a path from $x'$ to $x$ composed with $\gamma$, where $x'$ is the new basepoint. Thus, the map to $\Out(\pi_1)$ is independent of our choice of basepoints and paths between basepoints.
When $\phi$ is a homeomorphism, it is invertible, and thus $\phi_*$ is an automorphism. Hence, we have a well-defined homomorphism $\sigma: \Mod^\pm(S) \rarr \Out(\pi_1(S))$. \\

\begin{theorem}(Dehn-Nielsen-Baer)
    \textit{Let $g\ge1$ and let $S_g$ be a surface of genus $g$. The homomorphism}
    $$\sigma: \Mod^\pm(S_g) \rarr \Out(\pi_1(S_g))$$ 
    \textit{is an isomorphism.}
\end{theorem} 

A proof of the Dehn-Nielsen-Baer Theorem can be found in \cite[Chapter 8]{primer}.

\subsection{For bounded surfaces}
The following theorem is an extension of the 
Dehn-Nielsen-Baer Theorem to bounded surfaces where homeomorphisms are not required to fix the boundary pointwise. A proof can be found in \cite[Section 5.7]{zieschang}. \\

\begin{theorem}\label{boundeddnb}
    \textit{An automorphism $\alpha$ of the fundamental group of a bounded surface $\Sigma$ is induced by a homeomorphism if and only if 
    $\alpha(s_i) = l_i s^{\varepsilon}_{r_i}l^{-1}_i$,
    $i = 1, ..., m$ where ${1 \cdots m}\choose{r_1 \cdots r_m}$ is a permutation and 
    $\varepsilon = \pm 1, l_i \in \gp$}.
\end{theorem} 

In Theorem \ref{boundeddnb}, the $s_i$ are simple closed curves based at $x_0$ that are freely homotopic to a boundary component,
as shown in Figure \ref{fig:1} below.
In Section \ref{proof}, we will discuss how we can adjust this 
theorem in regards to the mapping class group, which does require homeomorphisms to fix the boundary.

\begin{figure}[!ht]
    \includegraphics[width=1\linewidth,trim=50 100 50 50, clip]{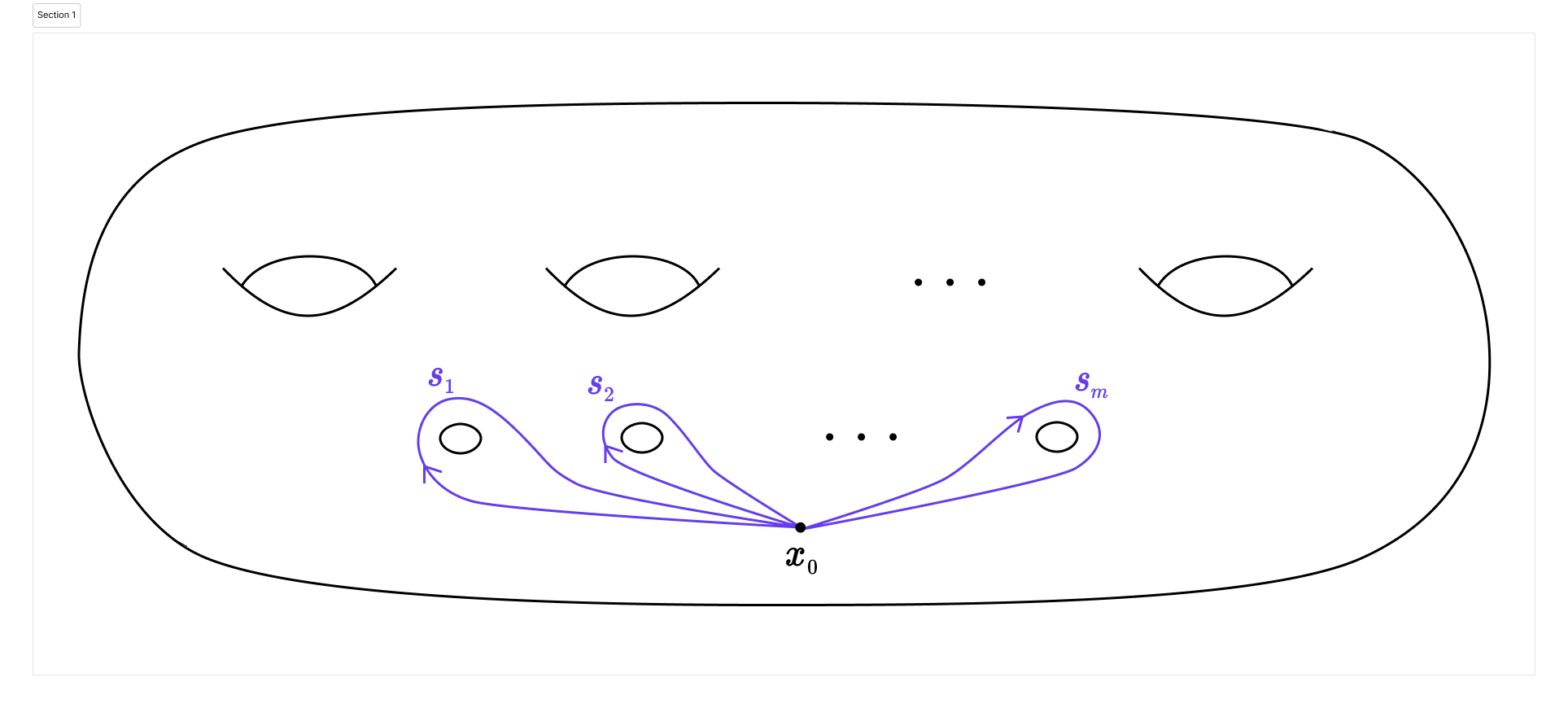}
    \caption{$s_i$ around a surface with $m$ boundary components.}
    \label{fig:1}
\end{figure}

\section{Groupoids and the fundamental groupoid}

A \textit{groupoid} $\mathcal G$ consists of a disjoint collection of sets
$\{G_{ij}\}_{i, j\in I}$ and an associative partial operation, 
$\cdot: G_{ij}\times G_{jk} \rarr G_{ik}$, which satisfy the following:
\begin{itemize}
    \item for each $i\in I$, there is an identity element, $e_i\in G_{ii}$, such that
    $e_i\cdot a = a$ and $b\cdot e_i = b$ for all $a\in G_{ij}$ and $b\in G_{ji}$
    \item for each $a\in G_{ij}$, there is an inverse element, $a^{-1}\in G_{ji}$, such that
    $a\cdot \inv{a} = e_i$ and $\inv{a}\cdot a = e_j$.
\end{itemize}
$I$ is called the \textit{object set} of $\mathcal G$ and when
$|I| = 1$, $\mathcal G$ is a group. In fact, $G_{ii}$ is a group 
for all $i\in I$, and we call these \textit{vertex groups}.
We also have \textit{source} and \textit{target} maps, 
$s, t: \mathcal G \rarr I$. For $g\in G_{ij}$, $s(g) = i$ and $t(g) = j$.
We say a groupoid is \textit{connected} if $G_{ij} \ne \emptyset$ 
for all $i, j\in I$, in which case we have
$G_{ii} \cong G_{jj}$ for all $i, j\in I$.\\ 

Fix some $i_0\in I$ and we define a \textit{star} to be 
$\{\iota_i\}_{i\in I}$,
where $\iota_i$ is a chosen element of $G_{i_0i}$ and 
$\iota_{i_0} = e_{i_0}$. 
Note that $\mathcal G$ is generated by $G_{i_0i_0}$ and $\{\iota_i\}_{i\in I}$,
 and we can uniquely write any element of $\mathcal G$ as 
$\iota_i^{-1}g\iota_j$ for some $g\in G_{i_0i_0}$.\\

\subsection{Automorphisms of groupoids\label{automorphisms}}
For a groupoid $\mathcal G$ with object set $I$, a \textit{homomorphism} $\phi$ of $\mathcal G$ is a function $\hat\phi:I\rarr I$ with functions
$\hat\phi_{ij}:G_{ij}\rarr G_{\hat\phi(i)\hat\phi(j)}$ for all
$i, j\in I$ such that:
\begin{itemize}
    \item $\hat\phi_{ij}(a)\hat\phi_{jk}(b) = \hat\phi_{ik}(ab)$ 
    for all $a\in G_{ij}$ and $b\in G_{jk}$
    \item $\hat\phi_{ii}(e_i) = e_{\hat\phi(i)}$ for all $i\in I$
    \item $\hat\phi_{ji}(g^{-1}) = \hat\phi_{ij}(g)^{-1}$
    for all $i, j\in I$ and $g\in G_{ij}$
    \item $s(\hat\phi_{ij}(g)) = \hat\phi(s(g))$ and $t(\hat\phi_{ij}(g)) = \hat\phi(t(g))$
    for all $g\in \mathcal G_{ij}$
\end{itemize}

An \textit{automorphism} of $\mathcal G$ is a homomorphism $\phi: \mathcal G\rarr \mathcal G$ with a two-sided inverse. The set of automorphisms of $\mathcal G$ forms a group under 
composition, which we denote as $\Aut(\mathcal G)$. 
Moving forward, we will simply use
$\phi$ to denote $\hat \phi$ and $\hat \phi_{ij}$.\\

A \textit{pure automorphism} of $\mathcal G$ is an 
automorphism of $\mathcal G$ that 
acts trivially on the object set. The set of all pure automorphisms forms a group called the  
\textit{pure automorphism group} of $\mathcal G$, and is denoted by 
$$\PAut(\mathcal G) = 
\{\phi\in \Aut(\mathcal G): \phi(i) = i \text{ for all }
i\in I\}.$$

There is a straightforward way to represent a pure
automorphism of a groupoid $\mathcal G$ with a finite object set. Let $\varphi$ be an 
automorphism of $\mathcal G$ so that $\varphi(i) = i$ for all $i\in I$. 
First, we pick a vertex group of $\mathcal G$,
$G = G_{0}$.
 Then we may write 
$\varphi = (\phi, (g_1, g_2, ..., g_n))$, where $\phi$
is an automorphism of $G$, $g_i\in G$, and $|I| = n+1$. 
For any element of the groupoid, $c\in G_{ij}$, we write 
$c = \inv{\iota_i} b \iota_j$ for some $b\in G$, and 
we have $$\varphi(c) = \inv{(g_i \iota_i)} \phi(b) (g_j \iota_j).$$
A rigorous proof can be found in \cite[Section 3]{AlpWensley2009}.\\ 

Furthermore, the following lemma tells us how to represent a pure automorphism of a groupoid with an infinite object set.\\

\begin{lemma} \label{infbaseptsisom}
\textit{Let $G$ be a vertex group of the connected groupoid $\mathcal G$. Then,}
$$(\Aut(G) \ltimes \prod_{i\in I}G) \slash \{(\Inn_g, (g^{-1}, g^{-1}, ...)) \mid g\in G\} \cong \PAut(\mathcal G)$$
\textit{where $g_i\in G$, $I$ is the (possibly infinite) object set, and $\Inn_g$ denotes conjugation by $g$.}\\
\end{lemma}

\textit{Proof.}
Fix a star $\{\iota_i \colon i\in I\} \setminus \{\iota_0\}$ based at $\iota_0$. Denote $\textbf{g} := (g_i)_{i\in I}\in \prod_{i\in I}G $, where $g_i \in G$ for all $i$.
We define the map $\Gamma: \Aut(G) \ltimes \prod_{i\in I}G \rarr \PAut(\mathcal G)$ by $(\psi, \textbf g) \mapsto \phi$ where $\phi(\iota_i^{-1}\alpha\iota_j) = g_i\iota_i^{-1}\psi(\alpha)g_j\iota_j$ and $\alpha$ is an element of $G$.
\\

Note that $\Gamma$ is well-defined since $\phi$ maps groupoid elements to groupoid elements, while preserving the starting and ending objects. Also, $\Gamma$ is a homomorphism since 
\begin{align*}
    (\phi'\circ\phi)(\iota_i^{-1}\alpha\iota_j) 
    &= \phi'(g_i\iota_i^{-1}\psi(\alpha)g_j\iota_j) \\
    &= g_i'g_i\iota_i^{-1}(\psi'\circ\psi)(\alpha)g_j'g_j\iota_j\\
    &= \phi'(\iota_i^{-1}\alpha\iota_j)\phi(\iota_i^{-1}\alpha\iota_j)
\end{align*}

It follows from \cite[Section 3]{AlpWensley2009} that $\Gamma$ is surjective.\\

We have that $\iota_0^{-1}\alpha \iota_0$, an element of $G$, is sent to
$\iota_0^{-1}g_0^{-1}\alpha g_0\iota_0 = g_0^{-1}\alpha g_0$ by any pure automorphism (since $\iota_0$ is the identity element). 
Additionally, $\iota_0^{-1}\alpha\iota_j$ gets sent to $\iota_0^{-1}g_0^{-1}\alpha g_j\iota_j = g_0^{-1}\alpha g_j\iota_j$ for all $j\ne 0$.
It follows that $\Gamma(\psi, \textbf g)$ is the identity automorphism if and only if $g_k = g_0$ for all $k$ and $\psi(\alpha) = g_0 \alpha g_0^{-1}$—that is, $\psi$ is the inner automorphism induced by $g_0$. Thus, the kernel of $\Gamma$ is $\{(\Inn_g, (g^{-1}, g^{-1}, ...)) \mid g\in G\}$. \\
\qed

\subsection{The fundamental groupoid}
One can think of the fundamental groupoid as the fundamental
group with multiple basepoints.
To be precise, let $X$ be a topological space and 
$I\sub X$ a subset. We define the \textit{fundamental groupoid}
$\pi_1(X, I)$ as the set of homotopy classes of paths
$\delta\colon [0, 1] \rarr X$ where $\delta(0), \delta(1)\in I$,
with concatenation as the groupoid operation and object set $I$.
Notice that when $I = \{x_0\}$, we simply have 
the fundamental group.
Furthermore, if $f: X\rarr X$ is a homeomorphism with
$f(I) = I$, then $f$ induces an automorphism 
$f_*\colon \pi_1(X, I) \rarr \pi_1(X, I)$ of 
the fundamental groupoid. 
If $f(x) = x$ for all $x\in I$, then 
$f_*$ is a pure automorphism.

\subsection{The fundamental groupoid with infinitely many basepoints.\label{infbasepoints}}\hfill\\
Let $\Sigma$ be a surface with $b$ boundary components, and let
$X' = \{x_{0}, x_{1}, ..., x_{b-1}\} \sub \partial \Sigma$, where each $x_i$ lies on a different boundary component.
For each boundary component, let $\gamma_i$ be a loop with basepoint $x_{i} \subset \partial\Sigma$ 
so that $\gamma_i$ is on the boundary component containing $x_{i}$.
Let $\gamma_{i}$ be oriented such that the surface is to the left of $\gamma_{i}$.
We may want to extend $X'$ to a (possibly infinite) set of basepoints such that there is more than one basepoint on some boundary component.  
In this case, we'll parametrize each $\gamma_i$ as a closed path and denote each basepoint as $x_{i, \alpha}, \alpha \in [0, 1)$, so that $x_{i, \alpha}$ is the point $\gamma_i(\alpha)$.
Let $\gamma_i(\alpha, \beta)$ denote the curve along the path $\gamma_i$ starting at $\gamma_i(\alpha)$ and ending at $\gamma_i(\beta)$ and let $\gamma_i$ denote $\gamma_i(0,1)$.
Let $X = \{x_{i, \alpha} : x_{i, \alpha} \text{ is a basepoint}\}$ denote the set of all basepoints. Note that $X' \sube X$, where $x_i \in X'$ is identified with $x_{i, 0} \in X$.\\

We'll define a generating set for $\gpoid$ as follows:
$$A = \{t_1, u_1, ..., t_g, u_g, \iota_{0,0}, \iota_{0,\alpha_0}, \iota_{0, \alpha_1}, ..., 
\iota_{1,0}, \iota_{1,\alpha_0}, ..., 
\iota_{b-1, 0}, \iota_{b-1, \alpha_0}, ..., \gamma_0, ..., \gamma_{b-1}\}$$
where:
\begin{itemize}
\item $\iota_{0,0}$ is represented by the identity loop based at $x_{0,0}$,
\item $\iota_{i, 0}$ is represented by a simple arc from $x_{0, 0}$ to $x_{i, 0}$,
\item $\iota_{i, \alpha_j}$ for $0 < \alpha_j < 1$ is $\iota_{i, 0}$ composed with $[\gamma_i(0, \alpha_j)]$,
\item and all other elements in $A$ are simple closed curves.
\end{itemize}
The elements of $A$ are shown in Figure \ref{fig:2}.
Let $G$ be the vertex group based at $x_{0, 0}$ and consider the star based at $x_{0, 0}$, 
$\{\iota_{i, \alpha}\}_{i\in X, \alpha\in [0,1)}$, described above and shown in Figure \ref{fig:2}.\\

We will define another type of automorphism of the fundamental groupoid. Let the set of \textit{boundary-constant automorphisms} of the fundamental groupoid be
$$
\BCAut(\gpoid) = \{\phi\in \Aut(\gpoid) : \phi([\gamma_i(0, \alpha)]) = [\gamma_i(0,\alpha)], \alpha\in [0,1]\}
$$

\begin{figure}[!ht]
\includegraphics[width=1\linewidth,trim=50 100 50 50, clip]{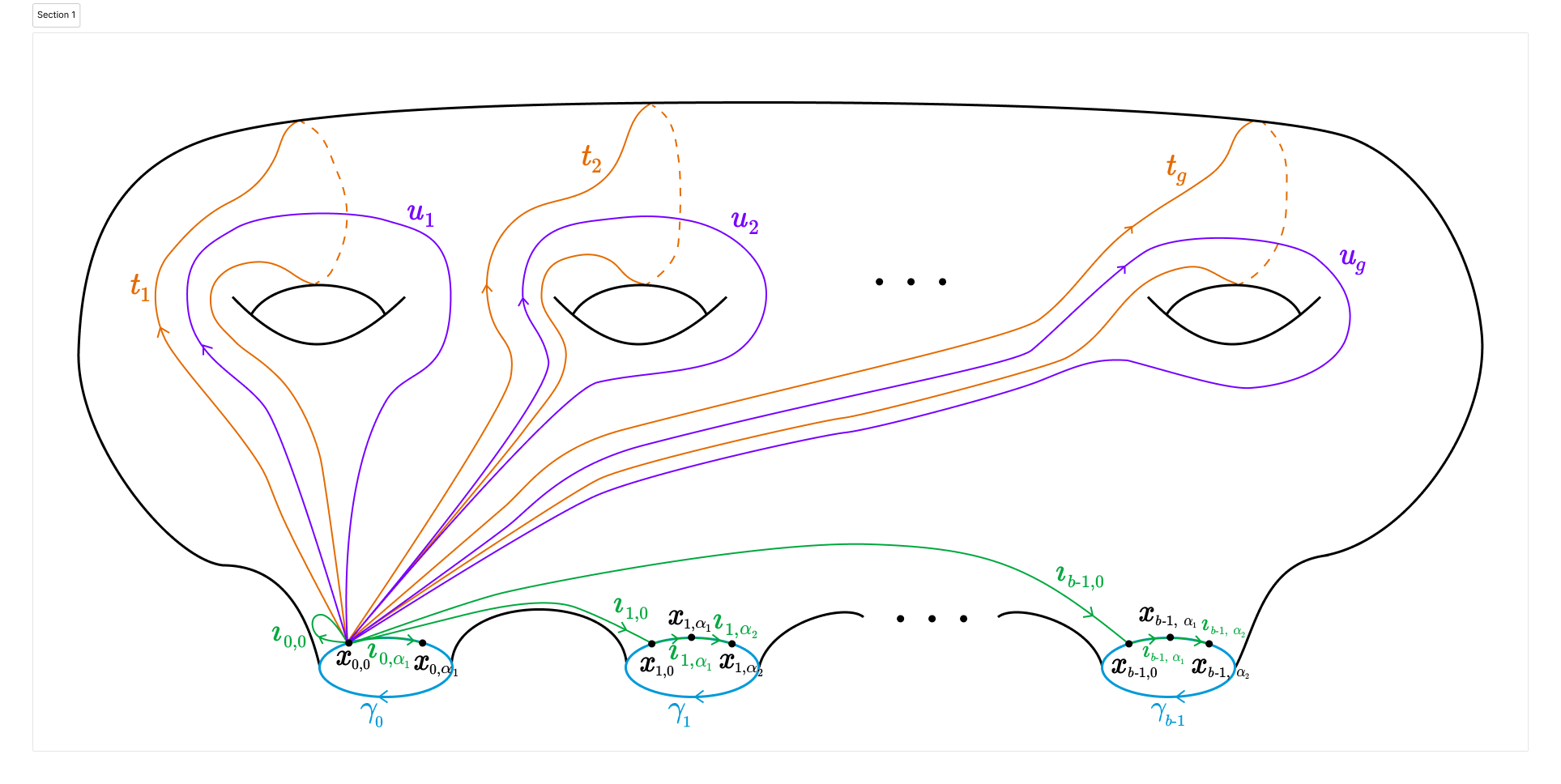}
\caption{A generating set for $\gpoid$}
\label{fig:2}
\end{figure}

For brevity, let $x_0$ denote $x_{0,0}$ and let $\iota_i$ denote $\iota_{i,0}$ from now on. By Lemma \ref{infbaseptsisom}, we have that we can represent a pure automorphism of a groupoid as an element of 
$(\Aut(G) \ltimes \prod_{i\in I}G) \slash \{(\Inn_g, (g^{-1}, g^{-1}, ...)) \mid g\in G\}$. The following lemma shows that we can do better for a boundary-constant automorphism of the fundamental groupoid.\\

\begin{lemma}\label{bcautdesc}
\textit{Let $\phi \in \BCAut(\gpoid)$. Let $\psi \Aut(\gp)$ and
$$h_{0,0}, h_{0, \alpha_1}, ..., h_{1, 0}, h_{1, \alpha_1}, ..., h_{b-1, 0}, h_{b-1, \alpha_1}, ... \in \gp$$ such that
$$\Gamma(\psi, (h_{0,0}, h_{0, \alpha_1}, ..., h_{b-1, 0}, h_{b-1, \alpha_1}, ...)) = \phi,$$ 
where} $\Gamma: \Aut(G) \ltimes \prod_{i\in I}G \rarr \PAut(\mathcal G)$ \textit{is the homomorphism from the proof of Lemma \ref{infbaseptsisom} defined by the star given at the beginning of Section \ref{infbasepoints} and shown in Figure \ref{fig:2}. Then, 
$h_{i, 0} = h_{i, \alpha}$ for all $0\le i < b$ and all $\alpha\in [0,1)$.}
\end{lemma}
\proof
Let $\varphi \in \BCAut(\gpoid)$ and consider $x_{i, 0}$. Pick a basepoint on the same boundary component, $x_{i, \alpha}$, and let
$\varphi(\iota_{i}) = h_{i, 0}\iota_i$ and
$\varphi(\iota_{i, \alpha}) = h_{i, \alpha}\iota_{i,\alpha}$, for some $h_{i, 0}, h_{i, \alpha} \in G$.
Note that $\iota_{i, \alpha} = \iota_{i}[\gamma_i(0, \alpha)]$, and so $\varphi(\iota_{i, \alpha}) = \varphi(\iota_{i}  )[\gamma_i(0, \alpha)]$ since $\varphi$ must fix $\gamma_i(0, \alpha)$.
Thus, $h_{i,\alpha}\iota_i
= h_{i,0}\iota_i[\gamma_i(0,\alpha)]
= h_{i,0}\iota_{i,\alpha}$,
which implies 
$h_{i,0} = h_{i,\alpha}$.\\
\qed\\

Lemmas \ref{infbaseptsisom} and \ref{bcautdesc} imply that an element $\varphi\in \BCAut(\gpoid)$ is completely determined by $(\psi, (h_1, \dots, h_{m-1}))$, where $\psi\in \Aut(\gp)$ and $h_i \in \gp$, with $\varphi(\iota_{j, \beta}^{-1}z\iota_{i, \alpha})
= (h_j\iota_{j,\beta})^{-1}\psi(z)(h_i\iota_{i,\alpha})$ for $z\in \gp$.


\section{Proof of result\label{proof}}
We are now ready to state and prove the main result.\\

\noindent
\textbf{Theorem 1.}
\textit{Let} $\Sigma$ \textit{be a bounded surface with $b$ boundary components and let} $X\sub \partial\Sigma$ \textit{be a subset of the boundary with at least one point on each boundary component. For each boundary component, let $\gamma_i$ be a loop on the boundary component, based at a point in $X$. Then, the map}
$$\Phi: \Mod(\Sigma) \rarr \BCAut(\gpoid),$$
\textit{given by the action of the mapping class group on the fundamental groupoid, is an isomorphism}

\proof
Let $g$ denote the genus of $\Sigma$.
First, we'll define the map $\Phi$.
Let $[\phi] \in \Mod(\Sigma)$, with $\phi$ a representative. 
Then, $\Phi([\phi])$ is the automorphism of $\pi_1(\Sigma, X)$ 
defined by $\Phi([\phi])([\beta]) = [\phi(\beta)]$, where $[\beta]\in \gpoid$. \\

This map is well-defined, as we certainly do get an automorphism of $\pi_1(\Sigma, X)$.
Since elements of the mapping class group fix the boundary components pointwise,
we have that $\phi(\gamma_i(0, \alpha)) = \gamma_i(0, \alpha)$ for all $i$ and $\alpha\in [0,1]$. Thus,
$\phi_*(\gamma_{i,\alpha}) = \phi(\gamma_i(0, \alpha)) = \gamma_{i,\alpha}$ holds for all $i, \alpha\in [0,1]$, and $\Phi([\phi]) \in \BCAut$. \\

Now, we'll show that $\Phi$ is injective via the Alexander Method \cite[Section 2.3]{primer}. 
Let $[f] \in \ker\Phi$ with $f$ a representative. 
Then, $\Phi([f]) = id$.\\

Consider the set $\{a,c_1, c_2, ..., c_{2g+1}, d_1, ..., d_{b-1}\}$ of curves and arcs in Figure \ref{fig:3} when $b>1$, and $\{a, c_1, c_2, ..., c_{2g+1}\}$ as shown in Figure \ref{fig:4} when $b=1$. 
We will select each arc $d_i$ such that it starts at the basepoint $x_i$ and ends at the basepoint $x_{i+1}$.
These collections satisfy the conditions to apply the Alexander Method. That is, the elements of these collections are essential simple closed curves and simple proper arcs that fill $\Sigma$, are pairwise in minimal position, and are pairwise non-isotopic. Also, for any three distinct elements, at least one pairwise intersection is empty.\\

\begin{figure}[!ht]
    \includegraphics[width=1\linewidth,trim=50 100 50 50, clip]{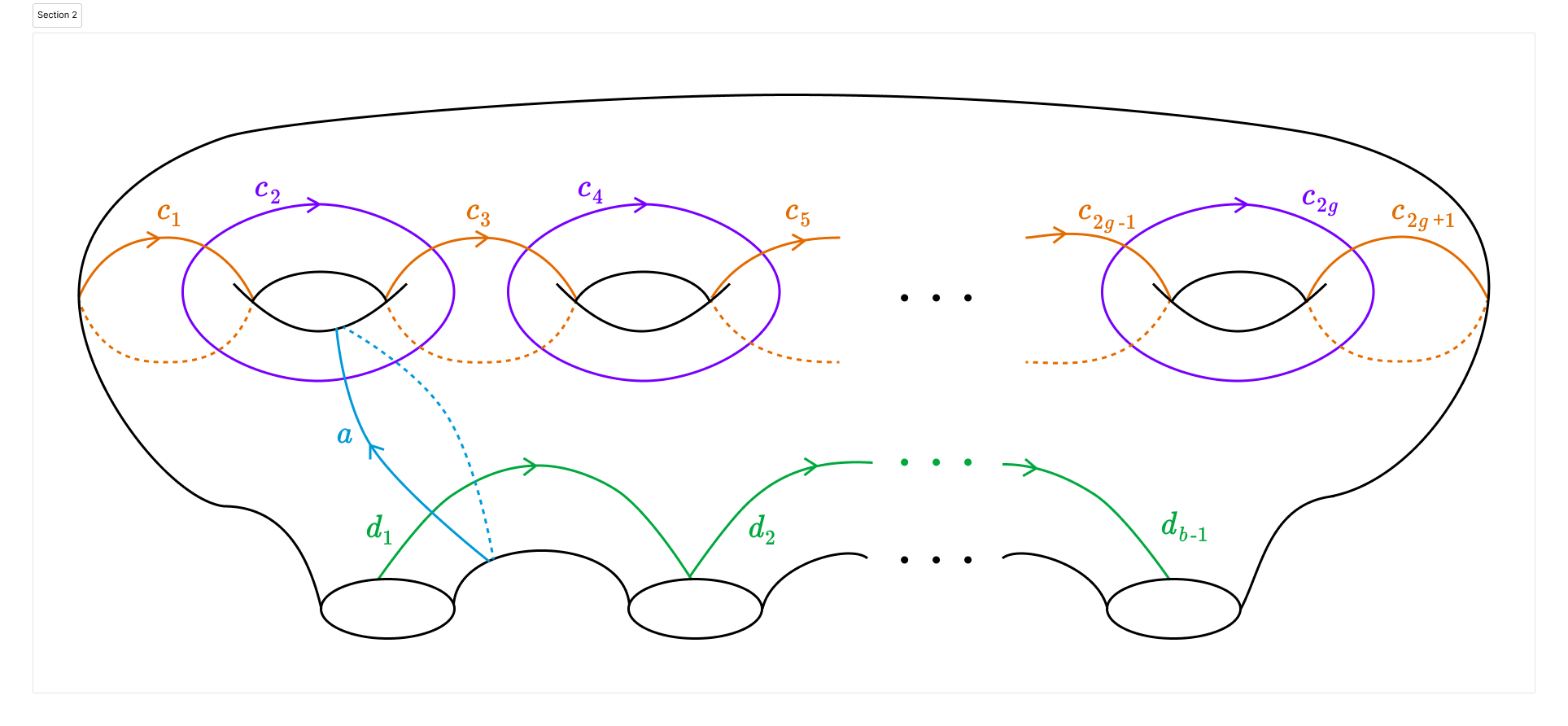}
    \caption{Collection of curves for the Alexander method for a surface with multiple boundary components}
    \label{fig:3}
\end{figure}

\begin{figure}[!ht]
    \includegraphics[width=1\linewidth,trim=50 100 50 50, clip]{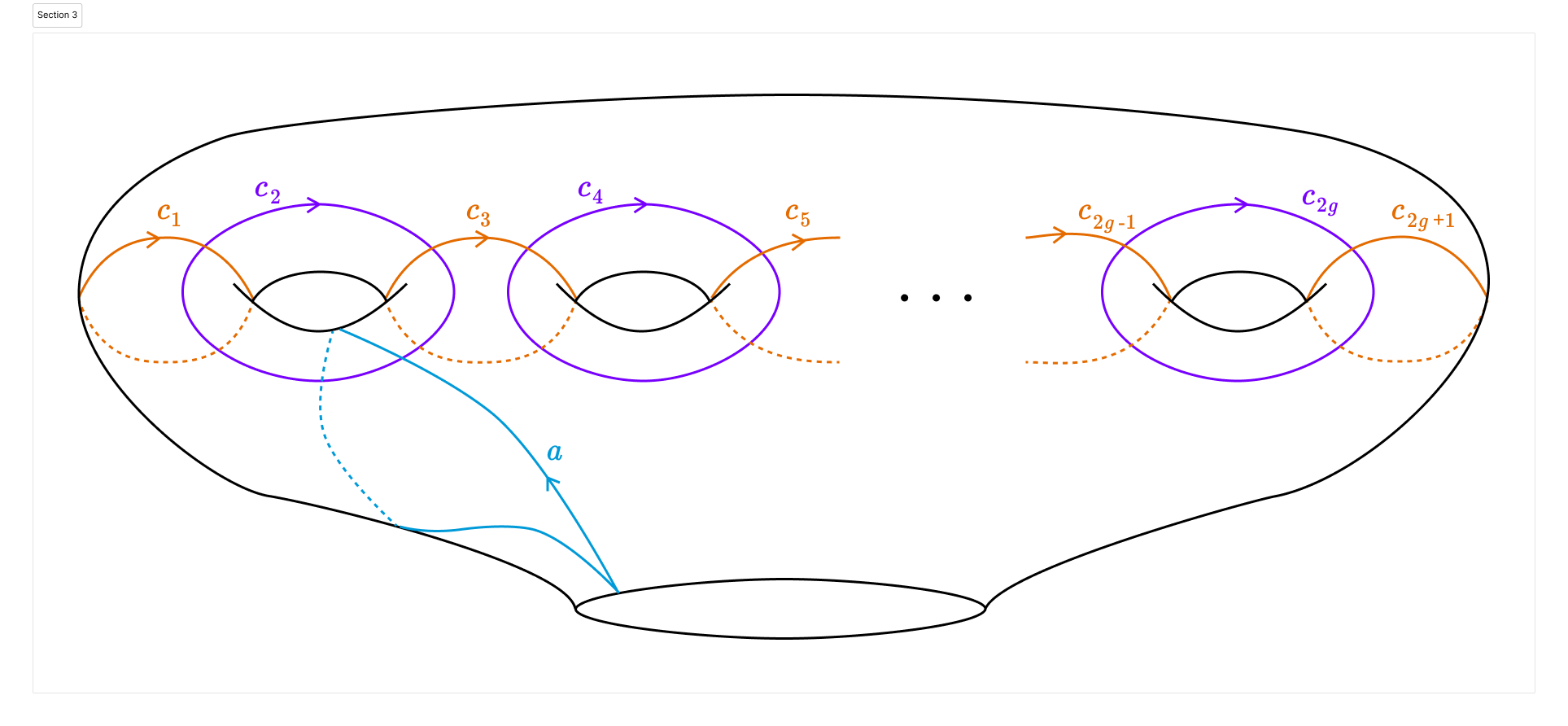}
    \caption{Collection of curves for the Alexander method for a surface with one boundary component}
    \label{fig:4}
\end{figure}

We can take any element $\alpha$ of the collection and naturally get an element of the fundamental groupoid that is freely homotopic to $\alpha$ by taking a path $\delta$ from $x_0$ to the curve. Then, $\delta \alpha\delta^{-1}\sim \alpha$ with $[\delta\alpha\delta^{-1}] \in \gpoid$,
where $\sim$ denotes the equivalence relation of free homotopy relative to $\partial\Sigma$.
Since $f$ acts trivially on the fundamental groupoid, we have that
$\alpha \sim \delta\alpha\delta^{-1} \sim 
f_*([\delta\alpha\delta^{-1}]) \sim f(\alpha)$,
where $f_*$ is the automorphism of the fundamental groupoid induced by $f$.
It follows that $f$ is isotopic to the identity.
Thus, $[f] = [id]$, and $\Phi$ is injective. \\

It remains to show that $\Phi$ is surjective. 
First, fix an orientation on $\Sigma$ such that the induced orientation of the boundary components is such that the surface is to the left of the boundary.
Select the same set of generating curves, $A$, and the same star as in Section \ref{infbasepoints}.
Take $\varphi \in \BCAut(\gpoid)$. \\

Define $s_i = \iota_i\gamma_i\inv{\iota_i}$.
Notice that
\begin{align*}
    \varphi(s_i) &= \varphi(\iota_i\gamma_i\inv{\iota_i}) \\
    &= \varphi(\iota_i)\varphi(\gamma_i)\varphi(\inv{\iota_i}) \\
    &= h_i\iota_i\gamma_i\inv{\iota_i}\inv{h_i} \\
    &= h_is_i\inv{h_i}
\end{align*}
for some $h_i\in \pi_1(\Sigma, x_0)$. Thus, $\varphi$ preserves conjugacy classes of the $s_i$. \\

Since we're considering homeomorphisms of $\Sigma$
that fix the boundary components,
it must be that the orientation of $\Sigma$ is preserved. 
Recall Theorem 2.2.1, which tells us when an automorphism of the fundamental
group of a bounded surface is induced by a homeomorphism.
If we don't require boundary components to be fixed pointwise, then
we have that $\varphi$ is induced by a homeomorphism.
But, how do we get a homeomorphism
that induces an automorphism of $\gp$, where
$x_0\in \partial_0$? \\

Let $y$ be a basepoint on $\Sigma\setminus\partial\Sigma$,
 define a path $\delta$ from $x_0$
to $y$, and define $\psi$ to be an automorphism of $\pi_1(\Sigma, y)$ 
with $\psi(\nu) = [\delta^{-1}]\theta[\delta]$, where $\nu\in \pi_1(\Sigma, y)$
and $\theta = \varphi([\delta]\nu[\delta^{-1}])$.
Let $[\omega] = [\delta^{-1}]\gamma_0[\delta]$, where $\omega$ is a representative, and note that 
\begin{align*}
    \psi([\omega]) &= [\delta^{-1}]z[\delta] \text{ where } 
    z = \varphi([\delta\omega\delta^{-1}]) = \varphi(\gamma_0) = \gamma_0\\
    &= [\delta^{-1}]\gamma_0[\delta] \\
    &= [\omega]
\end{align*}
We can apply Theorem 2.2.1 here, and we get that $\psi$ is induced by a homeomorphism $f'$.
Now, we'll isotope $f'$ to get $f$ such that
$f(x) = x, f(y) = y, f|_{\partial\Sigma} = id,$ and $f_* = \psi$, where $f_*$ is 
the automorphism of $\pi_1(\Sigma, y)$ induced by $f$. 
Note that we are able to isotope $f'$ to $f$ and have 
$f|_{\partial\Sigma} = id$ because $f$ is orientation-preserving.\\

We would now like to answer the question: what does the automorphism of $\gp$ 
induced by $f$ look like? Let $f_\#$ denote this automorphism, which we define to be 
$f_\#(\chi) = [\delta] f_*([\delta^{-1}] \chi [\delta]) [\delta^{-1}]$, for $\chi\in \gp$.

Now, consider $f_\#(\gamma_0)$:
\begin{align*}
    f_\#(\gamma_0) &= [\delta]f_*([\delta^{-1}]\gamma_0[\delta])[\delta^{-1}]\\
    &= [\delta]f_*([\omega])[\delta^{-1}]\\
    &= [\delta]\psi([\omega])[\delta^{-1}]\\
    &= [\delta\omega\delta^{-1}]\\
    &= \gamma_0
\end{align*}

Notice that since $f_* = \psi$ and $\psi([\omega]) = [\omega]$, 
we have $[f(\omega)] = [\omega]$. It follows that
\begin{align*}
    [f(\delta)\delta^{-1}]\gamma_0[\delta f(\delta)^{-1}] 
    &= [f(\delta) \omega f(\delta)^{-1}] \\
    &= [f(\delta) f(\omega) f(\delta)^{-1}]\\
    &= [f(\delta\omega\delta^{-1})] \\
    &= f_\#([\delta\omega\delta^{-1}]) \\
    &= f_\#(\gamma_0) \\
    &= \gamma_0
\end{align*}

Thus, $[f(\delta)\delta^{-1}]\gamma_0 = \gamma_0[f(\delta)\delta^{-1}]$, 
which implies that $ [f(\delta)\delta^{-1}] = \gamma_0^k, ~ k\in \Z$, since 
$\gamma_0$ is represented by a simple closed curve, and hence, a primitive element of a free group. 
Let $T_{\gamma_0}$ denote a representative of the Dehn twist about a representative of $\gamma_0$ that does not touch the $0$-th boundary component.
It follows that we can compose $f$ with $T_{\gamma_0}^{-k}$
to get a homeomorphism $\phi$ that 
induces $\varphi\in \Aut(\gp)$ (see Figure \ref{fig:5}).\\

Now, there exists a homeomorphism isotopic to $\phi$ 
that fixes the $x_i$ on each boundary component. 
By the Isotopy Extension Property \cite{isotopyextension}, 
there exists a homeomorphism $\eta$ of $\Sigma$ isotopic to $\phi$ that fixes
the boundary components pointwise. Since $\eta$ is isotopic to $\phi$,
we have that $\eta$ still induces $\varphi$. \\

Now, we have $\eta_* \in \Aut(\gpoid)$ with 
$\eta_* = (\varphi, (h_{1, 1}, ..., h_{b-1, 1}, ...))$, 
for some $h_i \in \gp$. By the way we
have defined $\iota_{i, \alpha}$, we have 
$h_{i,j} = h_{i,1}$ for all $j$ by Lemma \ref{bcautdesc}, so we will denote
$h_{i, 1}$ as $h_i$.
We will show that $h_i = g_is_i^{-k_i}$ for $k_i \in \Z, g_i\in \gp$:
\begin{align*}
    \eta_*(s_i) &= \varphi(s_i) \text{, since } s_i \in \gp \\
    \eta_*(s_i) &= \eta_*(\iota_i\gamma_i\inv{\iota_i}) \\
    &= \eta_*(\iota_i)\eta_*(\gamma_i)\inv{\eta_*(\iota_i)} \\
    &= g_i\iota_i\gamma_i\inv{\iota_i}\inv{g_i} \\
    &= g_is_ig_i^{-1}
\end{align*}
Since $s_i$ is represented by a simple closed curve,
$s_i$ is primitive \cite[pp. 23-26]{primer}. Thus, we have
\begin{align*}
& \varphi(s_i) = g_is_ig_i^{-1} \\
\implies &g_is_i\inv{g_i} = h_is_i\inv{h_i} \\
    \implies &\inv{h_i}g_is_i = s_i\inv{h_i}g_i \\
    \implies &\inv{h_i}g_i = s_i^{k_i}, ~ \text{ where } k_i\in \Z.
\end{align*}
It follows that $h_i = g_is_i^{-k_i}$, as desired. \\

Our goal is now to try to adjust $\eta$ so that it fixes the boundary components pointwise,
but still induces $\varphi$. We will do this by adding some Dehn twists to ``undo" the $s_i^{-k_i}$.
Let $\bar{\gamma_i}$ be a representative of $\gamma_i$ such that $\bar{\gamma_i}$ does not touch 
the $i$th boundary component.
Let $T_i$ be the Dehn twist about $\bar{\gamma_i}$.
Then, let $\mu = \eta \circ T_1^{-k_1} \circ ... \circ T_{b-1}^{-k_{b-1}}$.

\begin{figure}[H]
    \includegraphics[width=1\linewidth,trim=50 100 50 50, clip]{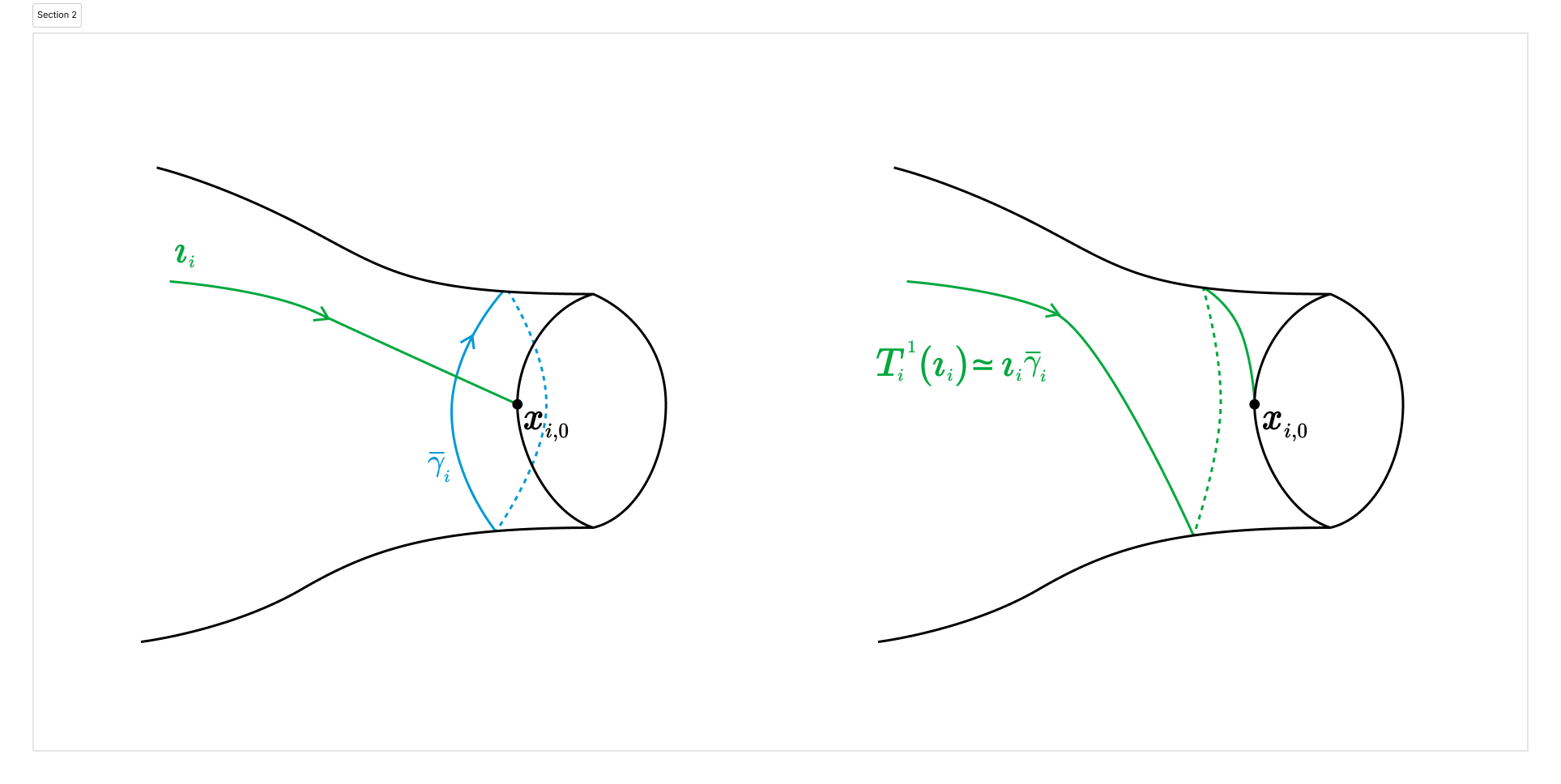}
    \caption{Dehn twist about $\bar\gamma_i$}
    \label{fig:5}
\end{figure}

$\mu$ certainly still induces $\varphi$, since $T_i^{-k_i}([a]) = [a]$ 
for any $a\in A$ by the way we have chosen the elements of the generating set $A$.
Notice that $s_i^{-k_i} 
= \iota_i{\gamma_i}^{-k_i}\inv{\iota_i}$.\\

Now, we have 
\begin{align*}
    \mu_*(\iota_i) &= \eta_*\circ T_{1}^{-k_1} \circ \cdots \circ T_{b-1}^{-k_{{b-1}}}(\iota_i) \\
    &= \eta_*\circ T_{i}^{-k_i}(\iota_i) \text{, since } \gamma_j \text{ (and thus } T_{j} \text{) is disjoint with } \iota_i \text{ when } i \ne j\\
    &= \eta_*(\iota_i\gamma_i^{-k_i}) \\
    &= g_i\iota_i\gamma_i^{-k_i}
\end{align*}
Note that $h_i = g_is_i^{-k_i} =
g_i\iota_i\gamma_i^{-k_i}\inv{\iota_i}$.
Thus, $h_i\iota_i = g_i\iota_i\gamma_i^{-k_i}$, and hence,
$\mu_*(\iota_i) = h_i\iota_i$.
It follows that $\mu$ fixes the boundary components pointwise and
induces $\varphi$.
Therefore, $\Phi$ is surjective, 
which completes the proof. \qed \\

\bibliographystyle{plain} 
\bibliography{references} 

\end{document}